\documentclass{amsart}
\usepackage{tikz}
\usepackage{xcolor}
\tikzstyle{vertex}=[inner sep = 0pt, minimum width=4pt, fill=black, shape=circle]

\newcommand{\gpoint}[2]{\node[style=vertex, label=#1:$#2$]}
\newcommand{\bpoint}[1]{\gpoint{below}{#1}}

\newcommand{\rpoint}[1]{\gpoint{right}{#1}}

\newtheorem{theorem}{Theorem}
\newtheorem{lemma}[theorem]{Lemma}
\newtheorem{corollary}[theorem]{Corollary}
\theoremstyle{remark}
\newtheorem{question}{Question}
\newtheorem{openquestion}{Open Question}
\newcommand{\floor}[1]{\left\lfloor #1 \right\rfloor}
\newcommand{\mth}{\Theta}
\title{Some results on Multithreshold Graphs}
\author{Gregory J.~Puleo}
\begin{document}
\begin{abstract}
  Jamison and Sprague defined a graph $G$ to be a \emph{$k$-threshold
    graph} with thresholds $\theta_1 , \ldots, \theta_k$ (strictly
  increasing) if one can assign real numbers $(r_v)_{v \in V(G)}$,
  called \emph{ranks}, such that for every pair of vertices $v,w$, we
  have $vw \in E(G)$ if and only if the inequality
  $\theta_i \leq r_v + r_w$ holds for an odd number of indices $i$.
  When $k=1$ or $k=2$, the precise choice of thresholds
  $\theta_1, \ldots, \theta_k$ does not matter, as a suitable
  transformation of the ranks transforms a representation with one
  choice of thresholds into a representation with any other choice of
  thresholds. Jamison asked whether this remained true for $k \geq 3$
  or whether different thresholds define different classes of graphs
  for such $k$, offering \$50 for a solution of the problem.  Letting
  $C_t$ for $t > 1$ denote the class of $3$-threshold graphs with thresholds $-1, 1, t$,
  we prove that there are infinitely many distinct classes $C_t$, answering
  Jamison's question. We also consider some other problems on multithreshold
  graphs, some of which remain open.
\end{abstract}
\maketitle
\section{Introduction}
Multithreshold graphs were introduced by Jamison and
Sprague~\cite{jamison-sprague} as a generalization of the well-studied
threshold graphs, first introduced by Chv\'atal and
Hammer~\cite{chvatal-hammer}. Given real numbers
$\theta_1, \ldots, \theta_k$ with
$\theta_1 < \theta_2 < \cdots < \theta_k$, we say that a simple graph
$G$ is \emph{a $k$-threshold graph with thresholds
$\theta_1, \ldots, \theta_k$} if there exist real numbers
$(r_v)_{v \in V(G)}$, called \emph{ranks}, such that for every pair of
distinct vertices $v,w \in V(G)$, we have $vw \in E(G)$ if and only if
the inequality $\theta_i \leq r_v + r_w$ holds for an odd number of
indices $i$. (Equivalently, adopting the convention that
$\theta_{k+1} = \infty$, we want $vw \in E(G)$ if and only if
$r_v + r_w \in [\theta_{2i-1}, \theta_{2i})$ for some $i$.)
In this case, we call $r$ a \emph{$(\theta_1, \ldots, \theta_k)$-representation of $G$}.

We will abbreviate this notation by saying that $G$ is
$(\theta_1, \ldots, \theta_k)$-threshold to mean that $G$ is
$k$-threshold with thresholds $\theta_1, \ldots, \theta_k$.
When $k=1$, we obtain the classical threshold graphs. 

In the case of the classical threshold graphs, it is clear that the
exact choice of threshold does not matter: by appropriately rescaling
the vertex ranks, any $\theta$-threshold graph is seen to also be a
$\theta'$-threshold graph.  The same observation holds for $k=2$: any
ranks witnessing that $G$ is $(\theta_1, \theta_2)$-threshold can be
transformed, via an appropriate affine transformation, into ranks
witnessing that $G$ is $(\theta'_1, \theta'_2)$-threshold.

At the 2019 Spring Sectional AMS Meeting in Auburn, Jamison asked whether
this phenomenon continues for higher $k$, and specifically whether it
still holds when $k=3$. Observing that an affine transformation of the
vertex ranks still uses up two ``degrees of freedom'' and let us
express any $(\theta_1, \theta_2, \theta_3)$-threshold graph as a
$(-1, 1, t)$-threshold graph for some $t$, his question can be phrased
as follows.
\begin{question}[Jamison]\label{qu:jamison}
  Do there exist real numbers $t, t' > 1$ such that the class of $(-1, 1, t)$-threshold graphs
  and the class of $(-1, 1, t')$-threshold graphs differ?
\end{question}
Jamison offered a \$50 bounty for an answer to this question. In this
paper, we answer the question in the affirmative: letting $C_t$ denote
the class of $(-1, 1, t)$-threshold graphs, we prove in Section~\ref{sec:11t} that there are
infinitely many distinct classes $C_t$.

We also study some other questions involving multithreshold graphs.
Say that $G$ is a \emph{$k$-threshold graph} if there exist real
numbers $\theta_1 < \cdots < \theta_k$ such that $G$ is a
$(\theta_1, \ldots, \theta_k)$-threshold graph.  Jamison and
Sprague~\cite{jamison-sprague} proved that for every graph $G$, there
is some $k$ such that $G$ is a $k$-threshold graph. Thus, we may
define the \emph{threshold number} $\mth(G)$ of a graph $G$
to be the smallest nonnegative $k$ such that $G$ is a $k$-threshold graph.

It is natural to compare the parameter $\mth(G)$ to other graph
parameters involving threshold graphs. Cozzens and
Leibowitz~\cite{cozzens} define the \emph{threshold dimension} $t(G)$
of a graph $G$ to be the smallest nonnegative integer $k$ such that
$G$ can be expressed as the union of $k$ threshold graphs.  Since the
complement of a threshold graph is a threshold graph, we can also view
$t(\overline{G})$ as the smallest nonnegative $k$ such that $G$ can be
expressed as the \emph{intersection} of $k$ threshold graphs.

Doignon observed, in a personal communication with the authors of
\cite{js-bithreshold}, that any $2$-threshold graph is the
intersection of two threshold graphs, hence $t(\overline{G}) \leq 2$
whenever $\Theta(G) \leq 2$.  This observation suggests a possible
converse:
\begin{question}[Jamison]\label{qu:2threshdim}
  Replacing $t(\overline{G})$ with $t(G)$, does $\Theta(G) \leq 2$ imply any bound on $t(G)$?  
\end{question}
\begin{question}[Jamison]\label{qu:thetat}
  Is $\Theta(G)$ bounded by any function of $t(\overline{G})$ or of $t(G)$?  
\end{question}
Question~\ref{qu:2threshdim} has a brief answer. For any graph $G$ and
positive integer $p$, let $pG$ be the disjoint union of $p$ copies of
$G$.  The graph $pK_2$ evidently has $t(G) = p$, since $2K_2$ is a
forbidden induced subgraph for a threshold graph; on the other hand,
$pK_2$ is a $(-1, 1)$-threshold graph, as witnessed by giving the
endpoints $u_i$ and $v_i$ of the $i$th edge ranks $r(u_i) = -2i$ and
$r(v_i) = 2i$. Hence there are graphs with $\Theta(G) = 2$ for which
$t(G)$ is arbitrarily large.

In Section~\ref{sec:threshdim}, we partially answer by proving that
there are graphs with $t(\overline{G}) = 3$ for which $\Theta(G)$ is
arbitrarily large. Finally, in Section~\ref{sec:openquestions}, we
discuss some remaining open problems about multithreshold graphs,
along the lines of the questions considered in this paper.

\section{Distinct families of $(-1, 1, t)$-threshold graphs}\label{sec:11t}
To facilitate proofs about multithreshold graphs, we introduce some
notational conventions. Given a multithreshold representation of a graph $G$,
the \emph{weight} of an edge or non-edge $uv$ is the sum of the ranks of $u$ and $v$.
When it is understood which multithreshold representation we are working with,
we will omit the function $r$ and simply write $v$ to stand for the
rank of the vertex $v$. (Hence, the weight of an edge $uv$ will simply
be written as $u+v$.)

For positive integers $p$, let $G_p = pK_2$. 
\begin{lemma}
  For any $p \geq 2$ and any $t > 2p-3$, the graph $G_p$
  is a $(-1, 1, t)$-threshold graph.
\end{lemma}
\begin{proof}
  Write $t = (1+2\epsilon)(2p-3)$ with $\epsilon > 0$.  Letting
  $a_i, b_i$ be the endpoints of the $i$th edge for
  $i=1, \ldots, p$, observe that the following ranks yield
  a $(-1,1,t)$-threshold representation of $G_p$:
  \begin{itemize}
  \item $a_i = -(1+\epsilon)(i-1)$ for $i = 1, \ldots, p$,
  \item $b_i = (1+\epsilon)(i-1)$ for $i = 1, \ldots, p$.
  \end{itemize}
  Evidently $a_i + b_i = 0$ for every edge $a_ib_i$. On the other
  hand, any nonadjacent pair of vertices has a weight whose
  absolute value is at least $1+\epsilon$, hence does not fall into
  the interval $[-1, 1)$, and whose value is at most
  $(1+\epsilon)(p-1) + (1+\epsilon)(p-2) = (1+\epsilon)(2p-3) < t$, hence
  does not fall into the interval $[t, \infty)$. Hence, this is
  a $(-1,1,t)$-representation of $G$.
\end{proof}
Computational experiments suggest that this bound is sharp:
that $G_p$ is not $(-1,1,t)$-threshold for any $t \leq 2p-3$.
Lacking a formal proof of this sharpness, we prove a weaker
statement.
\begin{lemma}\label{lem:gp-lower}
  For integer $p \geq 4$, if $G_p$ is a $(-1,1,t)$-threshold graph
  then $t > 2p-5$.
\end{lemma}
\begin{proof}  
  View the edges whose weight lies in $[-1,1)$ as colored red
  and view the edges whose weight lies in $[t, \infty)$ as colored yellow.
  Since the yellow edges form a threshold graph and $2K_2$ is a forbidden
  induced subgraph for threshold graphs, there is at most one yellow edge in $G_p$.
  Let $a_1b_1$, $a_2b_2$, \ldots, $a_pb_p$ be the edges of $G_p$.

  By symmetry, we may assume that $a_i \leq b_i$ for each $i$ and that
  $b_1 \leq \cdots \leq b_p$.  This implies that $b_p$ has the largest
  rank of all vertices and, thus, if there is a yellow edge, then
  that edge is $a_pb_p$.

  Let $q=p$ if $a_pb_p$ is red, and otherwise let $q=p-1$, so that all
  edges $a_1b_1, \ldots, a_qb_q$ are red.
  
  \textbf{Claim 1:} \emph{$a_k < a_j$ whenever $j < k \leq q$.} If not,
  then there exist $j < k$ with $a_k \geq a_j$ and $b_j \leq b_k$.
  Hence
  \[ a_k + b_j \geq a_j + b_j \geq -1, \]
  and
  \[ a_k + b_j \leq a_k + b_k < 1, \]
  which contradicts the fact that the edge $a_kb_j$ is absent. \bigskip

  It follows that the intervals $[a_i, b_i]$ are nested, with
  $[a_1, b_1] \subset [a_2, b_2] \subset \cdots \subset [a_{q}, b_{q}]$.

  \textbf{Claim 2:} \emph{$a_j + b_k \geq 1$ and $a_k + b_j < -1$
    whenever $j < k \leq q$.} Using the previous claim, we have
  \[ a_j + b_k \geq a_k + b_k \geq -1, \]
  hence $a_j + b_k \geq 1$ since otherwise the edge $a_jb_k$ should be present.
  Similarly, since $b_j \leq b_k$, we have
  \[ a_k + b_j \leq a_k + b_k < 1, \]
  hence $a_k + b_j < -1$ since otherwise the edge $a_kb_j$ should be present.\bigskip

  \textbf{Claim 3:} \emph{$b_j - a_j \geq 2(j-1)$ for all
    $j \in [q]$.}  We prove this by induction on $j$. When $j=1$ this
  is just the assumption that $b_j \geq a_j$. Assuming it holds for $j-1$,
  we prove that it holds for $j$. Observe that
  \[
    (b_j - a_j) - (b_{j-1} - a_{j-1}) = (a_{j-1} + b_j) - (a_j + b_{j-1}),
  \]
  and by the previous claim we have $a_{j-1} + b_j \geq 1$ and $a_j + b_{j-1} \leq -1$,
  so that
  \[ b_j - a_j \geq (b_{j-1} - a_{j-1}) + 2 \geq 2(j-2) + 2 = 2(j-1). \]
  \bigskip

  \textbf{Claim 4:} \emph{$b_j \geq j - 3/2$ for all $j \in [q]$.}
  This follows immediately from the inequalities
  \begin{align*}
    b_j - a_j &\geq 2j-2, \\
    b_j + a_j &\geq -1.
  \end{align*}\bigskip

  Having established these claims, we now complete the proof. If
  $q=p$, then Claim~4 gives $b_{p-1} \geq p - 5/2$ and
  $b_p \geq p - 3/2$, so to avoid the unwanted edge $b_{p-1}b_p$, it
  is necessary that $b_{p-1} + b_p < t$, which requires
  $(p-5/2) + (p-3/2) = 2p-4 < t$.

  If $q=p-1$, then Claim~4 gives $b_{p-1} \geq p - 5/2$, and since $a_p + b_p \geq t$
  with $a_p \leq b_p$, we have $b_p \geq t/2$. Hence, to avoid the unwanted edge
  $b_{p-1}b_p$ it is necessary that $(p-5/2) + t/2 < t$, which implies $2p-5 < t$.
\end{proof}
\begin{corollary}
  For each $k \geq 3$, the graph $G_{2^k}$ is $(-1,1,2^{k+1})$-threshold but not $(-1,1,t)$-threshold for any $t \leq 2^k$.
\end{corollary}
\begin{corollary}
  The classes $C_{2^k}$ for $k \geq 3$ are pairwise distinct.
\end{corollary}
\begin{corollary}
  For all $t > 1$, there exist $2$-threshold graphs that are not $(-1,1,t)$-threshold graphs.
\end{corollary}

\section{Threshold number versus threshold dimension}\label{sec:threshdim}
In this section, we partially answer Question~\ref{qu:thetat} by
proving that there exist graphs with $t(\overline{G}) = 3$ for which
$\Theta(G)$ is arbitrarily large. We will require the following result
of Cozzens and Leibowitz~\cite{cozzens} concerning the threshold
dimension of complete multipartite graphs:
\begin{theorem}[Cozzens--Leibowitz~\cite{cozzens}]\label{thm:cozzens}
  For positive integers $m_1 \leq \cdots \leq m_p$, the complete $p$-partite
  graph $K_{m_1,\ldots m_p}$ has threshold dimension $t(K_{m_1, \ldots, m_p}) = m_{p-1}$.
\end{theorem}
Let $G = pK_3$. Applying Theorem~\ref{thm:cozzens} with all $m_i = 3$ shows
that $t(\overline{G}) = 3$. Therefore, to show that $\Theta(G)$ can be arbitrarily
large for graphs with $t(\overline{G}) = 3$, it suffices to prove the following
theorem.
\begin{theorem}\label{thm:6p13}
  If $pK_3$ is a $k$-threshold graph, then $p \leq {k+2 \choose 3}$.
  In particular, $\mth(pK_3) \geq \frac{1}{2}(6p)^{1/3}$.
\end{theorem}
Note that the first part of the theorem is stronger than the second part,
which is obtained using a crude lower bound on ${k+2 \choose 3}$.

To prove this theorem, we will use a lemma stated in terms of edge
colorings (not necessarily proper) induced by a threshold representation.
Given a $(\theta_1, \ldots, \theta_k)$-representation of $pK_3$, we assign
colors $1, \ldots, \floor{k/2}$ to the edges of $pK_3$ by
giving edge $e$ color $i$ if its weight lies in the interval
$[\theta_{2i-1}, \theta_{2i})$.  (By the definition of a
$(\theta_1, \ldots, \theta_k)$-representation, for each edge there is
exactly one such $i$.)

Now, given an edge coloring, we can view each triangle as inducing a
multiset of colors on its edges (for example, we consider ``$2$ red
edges and $1$ yellow edge'' and ``$1$ red edge and $2$ yellow edges''
as different multisets, despite having the same underlying set).
\begin{lemma}\label{lem:rainbow}
  In a $(\theta_1, \ldots, \theta_k)$-representation of $pK_3$, no two triangles have the same multiset of colors appearing on their edges.
\end{lemma}
Before proving the lemma, we show how the proof of the theorem follows
immediately.
\begin{proof}[Proof of Theorem~\ref{thm:6p13}]
  When $p=1$, both parts of the theorem clearly hold, since $k \geq 1$ is required.
  For $p \geq 2$, observe that $pK_3$ has an induced $2K_2$ and thus is not a threshold
  graph; thus, we may assume $k \geq 2$.
  
  Assume $pK_3$ has a $(\theta_1, \ldots, \theta_k)$-representation. By Lemma~\ref{lem:rainbow}, no two triangles
  have the same multiset of colors on their edges. Hence, by the pigeonhole principle,
  the number of triangles is at most the number of size-$3$ multisets from $\{1, \ldots, k\}$,
  which by the standard stars-and-bars argument is ${k+2 \choose 3}$. Since $k \geq 2$, we have
  ${k+2 \choose 3} \leq {2k \choose 3} \leq (2k)^3/6$. Rearranging $p \leq (2k)^3/6$ gives the
  desired inequality.
\end{proof}
\begin{proof}[Proof of Lemma~\ref{lem:rainbow}]
  Suppose to the contrary that two triangles $x_1,x_2,x_3$
  and $y_1,y_2,y_3$ have the same multiset of colors on their edges.
  Without loss of generality, we may assume that:
  \begin{itemize}
  \item $x_1 \leq x_2 \leq x_3$,
  \item $y_1 \leq y_2 \leq y_3$, and
  \item $x_1 \leq y_1$.
  \end{itemize}
  Choose indices $\alpha, \beta, \gamma \in \{1, \ldots, k\}$ so that
  $x_1x_2$ has weight in $[\theta_{\alpha}, \theta_{\alpha+1})$, $x_1x_2$ has
  weight in $[\theta_{\beta}, \theta_{\beta+1})$, and $x_2x_3$ has weight in
  $[\theta_{\gamma}, \theta_{\gamma+1})$. (For convenience, we will adopt the convention that $\theta_{k+1} = \infty$.)  Observe that $x_1 \leq x_2 \leq x_3$
  forces $\alpha \leq \beta \leq \gamma$.

  Say an edge with weight in $[\theta_{\alpha}, \theta_{\alpha+1})$ is \emph{red},
  an edge with weight in $[\theta_{\beta}, \theta_{\beta+1})$ is \emph{yellow},
  and an edge with weight in $[\theta_{\gamma}, \theta_{\gamma+1})$ is \emph{pink}.
  (It is \emph{possible} that some of these thresholds coincide,
  in which case an edge may be, say, both red and yellow.)

  Since $y_1 \leq y_2 \leq y_3$ and the $y$-edges have the same multiset of
  colors as the $x$-edges, the colors of the $y$-edges must agree
  with the colors of the corresponding $x$-edges:
  \begin{itemize}
  \item $x_1x_2$ and $y_1y_2$ are red,
  \item $x_1x_3$ and $y_1y_3$ are yellow,
  \item $x_2x_3$ and $y_2y_3$ are pink.
  \end{itemize}
  Now we will derive our contradiction using the absence of the
  $x_iy_j$-edges.
  
  \textbf{Claim 1: $y_2 < x_2$.} If instead $x_2 \leq y_2$, then we
  have
  \[ \theta_{\alpha} \leq x_1 + x_2 \leq x_1 + y_2 \leq y_1 + y_2 < \theta_{\alpha+1}, \]
  forcing a red $x_1y_2$-edge, a contradiction. \medskip

  \textbf{Claim 2: $x_3 < y_3$.} If instead $y_3 \leq x_3$, then since
  $y_2 < x_2$, we have
  \[ \theta_{\beta} \leq y_2 + y_3 \leq y_2 + x_3 < x_2 + x_3 < \theta_{\beta+1}, \]
  forcing a pink $y_2x_3$-edge, a contradiction. \medskip

  Now since $x_3 < y_3$, we have
  \[ \theta_{\gamma} \leq x_1 + x_3 \leq y_1 + x_3 < y_2 + y_3 < \theta_{\gamma+1}, \]
  forcing a yellow $y_1x_3$-edge, again a contradiction. This
  completes the proof.
\end{proof}

\section{Remarks and Open Questions}\label{sec:openquestions}
After being informed of a preliminary version of the results in
Section~\ref{sec:11t}, Jamison (personal communication) suggested
studying the class $D = \bigcap_{t > 1}C_t$, where $C_t$ is the class
of $(-1,1,t)$-threshold graphs.

Intuition suggests that perhaps $D$ is related somehow to the class of
$2$-threshold graphs. Since $G_p$ is a $2$-threshold graph for all $p$,
Lemma~\ref{lem:gp-lower} implies that not all $2$-threshold graphs lie
in the class $D$. On the other hand, since all $2$-threshold graphs
satisfy $t(\overline{G}) \leq 2$, and since Theorem~\ref{thm:cozzens}
implies that $t(\overline{2K_3}) = 3$, we see that $2K_3$ is not a $2$-threshold graph;
however, $2K_3 \in D$, as the ranking in Figure~\ref{fig:2k3} can easily be verified
to be a $(-1, 1, t)$-representation for $2K_3$ whenever $\epsilon$ is sufficiently
small (in terms of $t$). Thus, $2K_3 \in D$ but $2K_3$ is not $2$-threshold;
the two classes are incomparable.
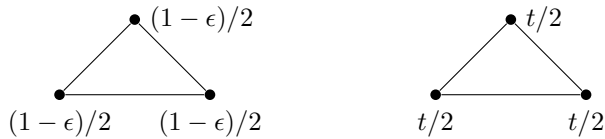
\begin{figure}
  \centering
    \begin{tikzpicture}
      \rpoint{(1-\epsilon)/2} (x1) at (1cm, 1cm) {};
      \bpoint{(1-\epsilon)/2} (y1) at (0cm, 0cm) {};
      \bpoint{(1-\epsilon)/2} (z1) at (2cm, 0cm) {};
      \begin{scope}[xshift=2cm]
        \rpoint{t/2} (x2) at (4cm, 1cm) {};
        \bpoint{t/2} (y2) at (5cm, 0cm) {};
        \bpoint{t/2} (z2) at (3cm, 0cm) {};        
      \end{scope}

      \draw (x1) -- (y1) -- (z1) -- (x1);
      \draw (x2) -- (y2) -- (z2) -- (x2);
    \end{tikzpicture}  
    \caption{$(-1,1,t)$-threshold ranking of $2K_3$.}
    \label{fig:2k3}
\end{figure}
\begin{openquestion}
  Is there a nice characterization of the class $D$?
\end{openquestion}
While the results in Section~\ref{sec:openquestions} imply that
there are at least countably many distinct classes $C_t$, it is not
clear whether there are countably many distinct classes or uncountably many distinct classes.
Indeed, it seems plausible that $C_t \neq C_{t'}$ whenever $t,t'$ are distinct real numbers
exceeding $1$.
\begin{openquestion}
  Are there uncountably many distinct classes $C_t$?
\end{openquestion}
\begin{openquestion}
  Are there distinct real numbers $t,t' > 1$ such that $C_t = C_{t'}$?
\end{openquestion}
Theorem~\ref{thm:6p13} only partially answers Question~\ref{qu:thetat},
which seeks a bound of $\Theta(G)$ in terms of $t(\overline{G})$ or $t(G)$. In particular,
the following questions remain open:
\begin{openquestion}
  Is $\Theta(G)$ bounded on the class of graphs $G$ with $t(\overline{G}) = 2$?
\end{openquestion}
\begin{openquestion}
  Is $\Theta(G)$ bounded by any function of $t(G)$?
\end{openquestion}

% Jamison asked what the containment relations are among the classes
% $C_t$.  The results in Section~\ref{sec:11t} imply that
% $C_{2^{k+1}} \not\subset C_{2^k}$ for $k \geq 2$, but it is not clear
% whether or not $C_{2^k}$ is contained in $C_{2^{k+1}}$.

% It is consistent with the results of Section~\ref{sec:11t} that
% $C_t \subset C_{t'}$ whenever $t < t'$. However, for each $t > 2$,
% the affine map $x \mapsto \frac{-2}{t-1}(x-1) + 1$ transforms
% a $(-1,1,t)$-representation of $G$ into a $(-1,1, 1 + \frac{4}{t-1})$-representation
% of $\overline{G}$. Thus, for example, $C_{8} \not\subset C_{4}$ implies that $C_{11/7} \not\subset C_{2}$,
% as $G_4 \in C_8 - C_4$ implies that $\overline{G_4} \in C_{11/7} - C_2$ via this affine map.

\section{Acknowledgments}
I thank Robert E.~Jamison both for posing the original
Question~\ref{qu:jamison} that motivated this paper as well as the
thought-provoking followup questions \ref{qu:2threshdim} and
\ref{qu:thetat} after receiving a preliminary version of these
results.

\bibliographystyle{amsplain}\bibliography{bibmult}
\end{document}